\documentclass[twoside,a4paper,12pt,reqno]{amsart}

\usepackage{amsmath,color,psfrag,graphicx}

\DeclareSymbolFont{euleroperators}{U}{eur}{m}{n}
\SetSymbolFont{euleroperators}{bold}{U}{eur}{b}{n}

\makeatletter
\renewcommand{\operator@font}{\mathgroup\symeuleroperators}
\makeatother

\definecolor{refkey}{rgb}{1,0,0}

\definecolor{labelkey}{rgb}{0,0,1}

\newtheorem{thm}[equation]{Theorem}

\newtheorem{cor}[equation]{Corollary}

\newtheorem{prp}[equation]{Proposition}

\theoremstyle{definition}
\newtheorem{dfn}[equation]{Definition}

\theoremstyle{remark}
\newtheorem{rem}[equation]{Remark}
\newtheorem{ex}[equation]{Example}

\newcommand{\thmref}[1]{Theorem~\ref{#1}}
\newcommand{\prpref}[1]{Proposition~\ref{#1}}

\newcommand{\corref}[1]{Corollary~\ref{#1}}
\newcommand{\dfnref}[1]{Definition~\ref{#1}}
\newcommand{\remref}[1]{Remark~\ref{#1}}
\newcommand{\exref}[1]{Example~\ref{#1}}

\newcommand{\secref}[1]{Section~\ref{#1}}

\renewcommand\a{\alpha}
\newcommand\Aff{\operatorname{Aff}}
\newcommand\Aut{\operatorname{Aut}}

\renewcommand\b{\beta}
\newcommand\bu{\bullet}
\newcommand\bubu{{\bullet\bullet}}

\newcommand\card{\operatorname{card}}

\newcommand\D{\Delta}

\newcommand\diag{\operatorname{diag}}
\newcommand\dist{\operatorname{dist}}

\newcommand\G{\Gamma}

\newcommand\Gr{\mathcal G}

\newcommand\I{\mathcal I}

\renewcommand\j{\iota}
\newcommand\jj{{\ddot{\iota}\hskip 1pt}}

\newcommand\M{{\mathrm M}}
\newcommand\m{\mathfrak{m}}

\newcommand\N{{\pmb{\sharp}}}
\newcommand\Nu{\operatorname{\pmb{\flat}}}

\renewcommand\o{\omega}
\newcommand\ov{\overline}

\newcommand\p{\partial}

\newcommand\q{\mathrel{\mathop:}=}

\newcommand\R{\mathcal R}
\renewcommand\r{\rho}

\newcommand\s{\sigma}

\newcommand\T{\mathcal T}
\renewcommand\th{\theta}

\newcommand\U{\mathcal U}

\newcommand\Z{\mathbb Z}

\begin{document}

\title[Invariance, quasi-invariance and unimodularity]{Invariance, quasi-invariance and unimodularity for random graphs}

\thanks{The author was partially supported by funding from the Canada Research Chairs
program, from NSERC (Canada) and from the European Research Council within European Union Seventh Framework Programme (FP7/2007-2013), ERC grant agreement 257110-RAWG. A part of this work was conducted during the trimester ``Random Walks and Asymptotic Geometry of Groups'' in 2014 at the Institut Henri Poincar\'e, Paris.}

\author{Vadim A. Kaimanovich}

\address{Department of Mathematics and Statistics, University of
Ottawa, 585 King Edward, Ottawa ON, K1N 6N5, Canada}

\email{vkaimano@uottawa.ca, vadim.kaimanovich@gmail.com}

\dedicatory{Dedicated to the memory of Mikhail Iosifovich Gordin}

\begin{abstract}
We interpret the probabilistic notion of unimodularity for measures on the space of rooted locally finite connected graphs in terms of the theory of measured equivalence relations. It turns out that the right framework for this consists in considering quasi-invariant (rather than just invariant) measures with respect to the root moving equivalence relation. We define a natural modular cocycle of this equivalence relation, and show that unimodular measures are precisely those quasi-invariant measures whose Radon--Nikodym cocycle coincides with the modular cocycle. This embeds the notion of unimodularity into the very general dynamical scheme of constructing and studying measures with a prescribed Radon--Nikodym cocycle.
\end{abstract}

\maketitle

\section*{Introduction}

Dealing with random infinite graphs inevitably requires to consider the associated probability measures. The natural state space for these measures is the space $\Gr_\bu$ of isomorphism classes of \emph{rooted ($\equiv$ pointed $\equiv$ marked) locally finite connected graphs} endowed with the projective topology (the space of unrooted graphs having no reasonable topological or measurable structure). Although there is no group action on $\Gr_\bu$, this space is endowed with the natural \emph{root moving equivalence relation} $\R$, so that one can talk about measures \emph{invariant} with respect to this equivalence relation. This notion of invariance for measures on $\Gr_\bu$ was introduced by the author \cite{Kaimanovich98}.

Invariance with respect to an equivalence relation also has an interpretation in terms of \emph{reversible Markov chains} on the state space of the equivalence relation. In the case of \emph{foliations} it is well-known that \emph{holonomy invariant measures} of Riemannian foliations are in one-to-one correspondence with reversible stationary ($\equiv$ \emph{completely invariant}, in the terminology of Garnett) measures of the \emph{leafwise Brownian motion}, see \cite{Garnett83,Kaimanovich88}. [With an abuse of language we talk here and below about ``reversible'' measures, although strictly speaking it is the Markov chain that is reversible with respect to the corresponding stationary measure; sometimes in this situation one says that the chain and the stationary measure are in \emph{detailed balance}, see \cite{Norris98}.] Equivalence relations are discrete counterparts of foliations, and for them this correspondence is even easier to establish as it directly follows from comparison of \emph{involution invariance} of the associated \emph{counting measure} (which is the definition of invariance of a measure with respect to an equivalence relation) with involution invariance of the joint distribution of the Markov chain at two consecutive moments of time (which is the definition of reversibility), see \cite{Kaimanovich98,Paulin99}. In what concerns unimodular measures, which we are going to discuss below, also see \cite{Aldous-Lyons07}, or a more recent paper \cite{Benjamini-Curien12}.

Lyons, Pemantle and Peres \cite{Lyons-Pemantle-Peres95} showed that the \emph{augmented Galton--Watson trees} provide a family of probability measures on the space $\T_\bu\subset\Gr_\bu$ of rooted trees which are stationary and reversible with respect to the simple random walk along the classes of the root moving equivalence relation. The author then noticed \cite{Kaimanovich98} that, in view of the above correspondence, this construction also produces \emph{$\R$-invariant measures} on $\T_\bu\subset\Gr_\bu$ (note that contrary to the claim in \cite[Example 1.1]{Aldous-Lyons07}, the paper \cite{Lyons-Pemantle-Peres95} contains no indications that reversibility of the measures associated with the augmented Galton--Watson trees might be related to any other invariance properties of these measures). Using the concept of invariance with respect to a graphed equivalence relation the author later defined the notion of \emph{stochastic homogeneity} for random graphs \cite{Kaimanovich03a}.

The idea of invariance with respect to an equivalence relation, although quite common in dynamical systems and geometry, is much less known in probability theory, and this approach has been mostly ignored by probabilists. On the other hand, by drawing their inspiration from the \emph{mass transfer} technique introduced by H\"aggstr\"om \cite{Haggstrom97} for studying percolation on
trees and later extended by Benjamini, Lyons, Peres, and  Schramm \cite{Benjamini-Lyons-Peres-Schramm99} (following reasoning of Woess \cite{Woess94}) to arbitrary graphs, Benjamini and Schramm \cite{Benjamini-Schramm01} defined an \emph{intrinsic mass transport principle} for probability measures on $\Gr_\bu$. The resulting class of measures has now become quite popular in the probabilistic world under the name of \emph{unimodular measures} introduced by Aldous and Lyons \cite{Aldous-Lyons07}. This popularity is mostly due to the following important property: the class of unimodular measures (unlike the class of invariant measures) is closed in the weak$^*$ topology, so that, in particular, any weak$^*$ limit of discrete unimodular measures corresponding to finite graphs is also a unimodular measure.

In the case of a single graph it was almost immediately realized in \cite{Benjamini-Lyons-Peres-Schramm99} that the mass transfer principle is nothing else than plain \emph{unimodularity of the group of automorphisms of the graph}. However, in what concerns unimodular measures, so far this notion has been mostly confined to the probabilistic world. Although probabilists, when talking about unimodularity, do not usually think in terms of the theory of measured equivalence relations (and moreover claim that ``probabilistic aspects'' are out of its focus of attention, see \cite[Section 9]{Aldous-Lyons07}), the state space $\Gr_\bu$ of a unimodular measure is, \emph{nolens volens}, endowed with the equivalence relation $\R$, so that it is natural to ask about the properties of unimodular measures with respect to this structure (in particular, $\R$-invariance and $\R$-quasi-invariance).

For measures concentrated on the set of \emph{rigid graphs} (i.e., those with trivial group of automorphisms, graph theorists also call such graphs fixed or symmetry-free \cite{Harary-Ranjan98, Harary01}) the notions of invariance and unimodularity obviously coincide; however, they do differ in general, as one can easily see by looking just at finite graphs. [Note that according to what we call \emph{Tutte's principle}, the graphs produced by ``natural'' probabilistic constructions should be almost surely rigid, see \remref{rem:tutte}.]

\vskip 5mm

The purpose of the present research communication is to clarify the relationship between $\R$-invariance (and $\R$-quasi-invariance), on one hand, and unimodularity, on the other hand, in full generality. We do it by embedding the notion of unimodularity into the general context of the ergodic theory of discrete measured equivalence relations, and by replacing the ``mass transfer principle'' with the technique based on explicit considerations of the $\s$-finite counting measures on the equivalence relation $\R$ and on the space of doubly rooted graphs $\Gr_\bubu$. Although in the present paper we only deal with \emph{probability} unimodular measures, this technique is applicable to $\s$-finite unimodular measures as well.

\textbf{Our main result is that unimodular measures are precisely those $\R$-quasi-invariant measures whose Radon--Nikodym cocycle coincides with a natural \emph{modular cocycle} (\thmref{thm:m}).} Therefore, from this point of view the notion of unimodularity embeds into the very general scheme of constructing and studying \emph{measures with a prescribed Radon--Nikodym cocycle} (cf. the notions of \emph{Gibbs, Patterson--Sullivan,} or \emph{conformal} measures \cite{Sinai72,Bowen75,Sullivan79,Kaimanovich-Lyubich05}).

As a byproduct of our approach we also obtain much shorter and more concise proofs of a number of results which earlier required lengthy calculations based on the ``mass transfer principle'' (see  \remref{rem:z}, \corref{cor:zz}, \remref{rem:nu}, \remref{rem:zzz}).

\vskip 5mm

Let us briefly outline the content. Given a single connected graph $\G$, we define its (multiplicative) \emph{modular cocycle} as the ratio
$$
\D_\G(x,y) = \frac{\card G_x y}{\card G_y x} \;,
$$
where $G$ is the group of automorphisms of $\G$, and $G_x y$ is the orbit of a vertex $y\in\G$ under the stabilizer of a vertex $x\in\G$ in $G$. For any $x\in\G$ the function $g\mapsto\D_\G(x,gx)$ coincides with the \emph{modular function} of the group $G$, which was known already to Schlichting \cite{Schlichting79} and Trofimov \cite{Trofimov85}, and it is not hard to see that $\D_\G$ satisfies the cocycle identity in the general case as well. If the group of automorphisms $G$ is unimodular, then $\D_\G$ descends to a cocycle $\D_{\G_\bu}$ on the quotient $\G_\bu$ of the graph $\G$ by the group $G$. Since classes of the equivalence relation $\R$ are precisely the quotient graphs $\G_\bu$, the individual cocycles $\D_{\G_\bu}$ piece together to a global \emph{modular cocycle} $\D_\bu$ of the equivalence relation $\R$ restricted to the subset $\breve{\Gr}_\bu\subset\Gr_\bu$ which consists of the equivalence classes $\G_\bu$ corresponding to graphs $\G$ with a unimodular group of automorphisms.

Our approach is based on a systematic use of the analogy between the equivalence relation $\R\subset\Gr_\bu\times\Gr_\bu$ and the \emph{space of doubly rooted graphs} $\Gr_\bubu$ which appears in the definition of unimodular measures. The space $\Gr_\bubu$ is ``bigger'' than $\R$ as there is a natural map $\s:\Gr_\bubu\to\R$ which assigns to any doubly rooted graph the pair of plain rooted graphs obtained by retaining just one of the two roots. Both $\R$ and $\Gr_\bubu$ are fibered over the space of rooted graphs $\Gr_\bu$ with at most countable fibers, and both are endowed with natural involutions $\j$ and $\jj$, respectively, see diagram \eqref{eq:d3}.

The classical construction of Feldman and Moore \cite{Feldman-Moore77} assigns to any measure $\mu$ on $\Gr_\bu$ the associated \emph{counting measure} $\M_\R$ on $\R$ as the result of integration of the fiberwise counting measures $\N_\xi,\,\xi\in\Gr_\bu,$ against the measure $\mu$. Then $\R$-invariance of the measure $\mu$ is equivalent to involution invariance of $\M_\R$. In the same way, the fibers of the projection $\Gr_\bubu\to\Gr_\bu$ are also endowed with a natural family of measures $\Nu_\xi,\,\xi\in\Gr_\bu,$ defined in equation \eqref{eq:bx} from \secref{sec:graphs}, and their integration against the measure $\mu$ gives the associated measure $\M_{\Gr_\bubu}$ on $\Gr_\bubu$ such that unimodularity of $\mu$ is equivalent to involution invariance of $\M_{\Gr_\bubu}$. The key property of the family $\{\Nu_\xi\}$ used in the proof of \thmref{thm:m} is the identity (\prpref{prp:dbx})
$$
\D_\bu (\xi,\eta) = \frac{\Nu_\xi(\th)}{\Nu_\eta(\jj\th)} \qquad\forall\,\th\in\s^{-1}(\xi,\eta) \;.
$$

Note that the original formulation of the intrinsic mass transport principle of Benjamini and Schramm in \cite{Benjamini-Schramm01} was given just in terms of the expectations of appropriate test ``transport functions'' with respect to the measure $\mu$ on $\Gr_\bu$, and contained the measure $\M_{\Gr_\bubu}$ only in an implicit form. Although explicitly this measure does appear once in \cite{Aldous-Lyons07} (in the explanation for the choice of the term ``unimodularity'' after Definition 2.1), apart from that the authors of \cite{Aldous-Lyons07} (as well as those of \cite{Benjamini-Schramm01}, or those of a more recent paper \cite{Benjamini-Curien12}) always deal with $\mu$-expectations, and never explicitly use either the counting measure $\M_{\Gr_\bubu}$, or the family of its fiberwise measures $\Nu_\xi$. In our approach we deal directly with measures (rather than with expectations with respect to these measures), which significantly simplifies and clarifies the exposition.

Using the measures $\M_\R$ and $\M_{\Gr_\bubu}$ one can also \emph{``quasify''} the notions of $\R$-invariance and unimodularity, respectively, by requiring that these measures be involution quasi-invariant. In the case of $\R$ this gives the usual definition of quasi-invariance of the measure $\mu$ with respect to the equivalence relation $\R$. It turns out that quasi-unimodularity is equivalent to quasi-invariance. Our main technical tool is \thmref{thm:main} which connects the Radon--Nikodym derivative of the counting measure $\M_{\Gr_\bubu}$ with respect to the involution on $\Gr_\bubu$ with the Radon--Nikodym cocycle of the measure $\mu$.

As an illustration of the advantages of our approach we almost immediately show that $\R$-ergodic unimodular measures are precisely the extreme ones in the convex set of all unimodular measures, and that the decomposition of a unimodular measure into ergodic components with respect to the equivalence relation $\R$ coincides with its decomposition into an integral of extreme unimodular measures (\thmref{thm:ex} and \thmref{thm:dec}). We also give an explicit description of the \emph{Hopf decomposition} of a general unimodular measure into dissipative and conservative parts (\thmref{thm:hopf}).

\vskip 5mm

In the present paper we have tried to provide general dynamical, geometric and algebraic contexts for the considered probabilistic notions. By skipping secondary technicalities, we emphasised the underlying concepts and the examples which illustrate them. A more detailed exposition will appear elsewhere.

\vskip 5mm

On numerous occasions I had the privilege to discuss these as well as other related and not so related topics with Mikhail Iosifovich. Random trees never prevented him from seeing blooming gardens and splendid vistas of the great world of mathematics and beyond.

\section{Modular cocycle}

\subsection{Graphs, their quotients, and measures on them} \label{sec:graphs}

Let $\G$ be a \textsf{locally finite connected graph without loops or multiple edges}. As usual, the same notation $\G$ will also be used for the \textsf{vertex set} of this graph, so that the \textsf{edge set} of $\G$ is just a symmetric subset of $\G\times\G\,\setminus\diag$. By $\N_\G$ we denote the \textsf{counting measure on} $\G$, i.e.,
$$
\N_{\G}(A) \q \card A \;, \qquad A\subset \G \;,
$$
and by
$$
G\q\Aut(\G) \;,\qquad G_x \q \{g\in G: gx=x\} \;,
$$
we denote the \textsf{group of automorphisms of} $\G$ and the $G$\textsf{-stabilizer} of a vertex $x\in\G$, respectively. The group $G$ is \emph{locally compact} with respect to the pointwise convergence topology.

By
\begin{equation} \label{eq:Gbu}
\G_\bu\q\G/G = \left\{ \ov x\q Gx: x\in\G \right\}
\end{equation}
we denote the associated \textsf{orbital quotient} of $\G$, i.e., the set of $G$-orbits in $\G$. Equivalently, $\G_\bu$ is the set of \textsf{isomorphism classes} $\ov{(\G,x)}$ \textsf{of} all possible \textsf{rootings} $(\G,x)$ of the graph $\G$. By $\N_{\G_\bu}$ we denote the \textsf{counting measure on} $\G_\bu$.

The orbital quotient $\G_\bu$ has a natural \textsf{graph structure}: two orbits $\ov x,\ov y\in\G_\bu$ are neighbours if the \textsf{graph distance} $\dist$ in $\G$ between them is equal to 1, i.e., if for any ($\equiv$ certain) representative $x\in\ov x$ there is a representative $y\in\ov y$ such that $x$ and $y$ are neighbours in~$\G$.

Further, let
\begin{equation} \label{eq:Gax}
\G^x_\bu \q \G/G_x = \left\{ \ov{y}^x \q G_x y: y\in\G \right\}   \;, \qquad x\in\G \;,
\end{equation}
be the \textsf{set of $G_x$-orbits in} $\G$. When talking in terms of rooted graphs we shall use the notation $\ov{(\G,y)}^x\in\G^x_\bu$. By $\Nu_x$ we denote the \textsf{measure on} $\G^x_\bu$ which is the image of the counting measure $\N_\G$ under the map $y\mapsto\ov{y}^x$ from $\G$ to $\G^x_\bu$. In other words,
\begin{equation} \label{eq:bx}
\Nu_x \left( \ov{y}^x \right) \q \card G_x y \;,
\end{equation}
which is finite because of local finiteness of $\G$.

Finally, let
$$
\G_\bubu\q (\G\times\G)/G = \left\{ \ov{(x,y)}\q G(x,y): (x,y)\in\G\times\G  \right\}
$$
be the \textsf{set of $G$-orbits on the square of the graph} $\G$. Equivalently, $\G_\bubu$ is the set of \textsf{isomorphism classes} $\ov{(\G,x,y)}$ \textsf{of} all \textsf{double rootings} of $\G$, i.e., of isomorphism classes of the graph $\G$ endowed with two distinguished vertices: \textsf{primary} and \textsf{secondary roots} $x$ and $y$, respectively.

\vfill\eject

\subsection{Projections and their fibers} \label{sec:diagrams}

There are several natural maps between the sets $\G_\bubu, \G_\bu\times\G_\bu$, and $\G_\bu$, whose definitions should be clear from the following pretty straightforward commutative diagram:
\begin{equation} \label{eq:d1}
\begin{split}
          \includegraphics{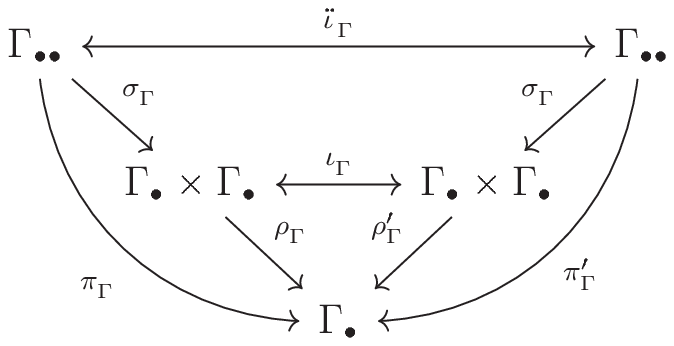}
\end{split}
\end{equation}

\begin{equation} \label{eq:d2}
\begin{split}
          \includegraphics{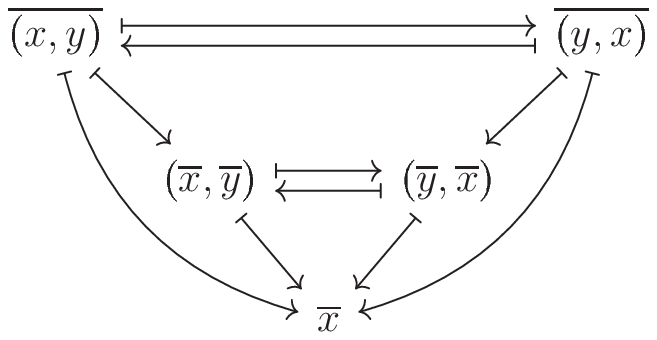}
\end{split}
\end{equation}

We shall endow the \textsf{fibers}
$$
\r_\G^{-1}(\xi) \cong \G_\bu \;,\qquad \xi\in\G_\bu \;,
$$
with the \textsf{measures} $\N_\xi\,$, which \textsf{are the images of the counting measure} $\N_{\G_\bu}$ \textsf{under the natural identification of} $\r_\G^{-1}(\xi)$ \textsf{and} $\G_\bu$.

As for the \textsf{fibers}
$$
\G_\bu^\xi \q \pi_\G^{-1}(\xi) \;,\qquad \xi\in\G_\bu \;,
$$
they can be identified with the quotients $\G_\bu^x$ \eqref{eq:Gax} for the vertices $x$ from the $G$-orbit $\xi$ (i.e., for such $x\in\G$ that $\ov{x}=\xi$). Under this identification the measures $\Nu_x$ \eqref{eq:bx} produce a \textsf{measure on} $\G_\bu^\xi$ which does not depend on $x$, and which we shall denote by $\Nu_\xi$. In other words,
\begin{equation} \label{eq:bxy}
\Nu_{\ov x} \left( \ov{(x,y)} \right) \q \card G_x y \;.
\end{equation}

\begin{rem} \label{rem:GGu}
The map $\s_\G$ from diagram \eqref{eq:d1} is a bijection if and only if the graph $\G$ is \textsf{rigid} (i.e., its automorphisms group $G$ is trivial; graph theorists also call such graphs \textsf{fixed} or \textsf{symmetry-free} \cite{Harary-Ranjan98,Harary01}), and in this situation the measures $\Nu_\xi$ coincide with the corresponding measures $\N_\xi$, i.e., just with the counting measure $\N_\G$. Indeed, $\s_\G$ is a bijection if and only if any orbit $G(x,y)$ in $\G\times\G$ splits as a product of the orbits $Gx$ and $Gy$. In particular, any diagonal orbit $G(x,x)$ should also split this way, which is only possible if $G$ is trivial. On the other hand, in the general case the fibers of $\s_\G$ can be quite big. The extreme example is provided by infinite \textsf{vertex transitive graphs}, for which $\G_\bu$ is a singleton, whereas $\G_\bubu$ is infinite.
\end{rem}

\subsection{Modular cocycle and its properties}

\begin{prp}
The ratio
\begin{equation} \label{eq:D}
\D_\G(x,y) \q \frac{\card G_x y}{\card G_y x} \;,\qquad x,y\in\G \;,
\end{equation}
satisfies the (multiplicative) \textsf{cocycle identity}
$$
\D_\G(x,y)\D_\G(y,z)=\D_\G(x,z) \qquad\,\forall\,x,y,z\in\G \;.
$$
\end{prp}

\begin{proof}
By definition, $\card G_x y$ coincides with the (left) index of the \textsf{joint stabilizer}
$$
G_{xy}=G_x\cap G_y
$$
in $G_x$, or, equivalently, with the ratio of the corresponding values of the \textsf{left Haar measure} $m$ on $G$:
$$
\card G_x y = \frac{m (G_x)}{m(G_x\cap G_y)} \;,
$$
whence
\begin{equation} \label{eq:DH}
\D_\G(x,y) = \frac{m(G_x)}{m(G_y)} \;,
\end{equation}
so that it obviously satisfies the cocycle identity.
\end{proof}

\begin{dfn} \label{dfn:D}
The ratio $\D_\G$ (\ref{eq:D}) is called the \textsf{modular cocycle} of the graph $\G$.
\end{dfn}

\begin{rem}
For vertex transitive graphs this cocycle (in a somewhat implicit form) appeared already in the papers of Schlichting \cite[Lemma 1]{Schlichting79} and Trofimov \cite[Theorem 1]{Trofimov85}. Explicitly (again, just for vertex transitive graphs) the cocycle $\D_\G$ was for the first time defined by Soardi and Woess \cite[Lemma 1]{Soardi-Woess90} (also see the exposition in \cite[Section~1.F]{Woess00}).
\end{rem}

\begin{dfn} \label{dfn:dUG}
A graph $\G$ is called \textsf{unimodular} if its modular cocycle is identically 1.
\end{dfn}

\begin{rem}
Our terminology is different from that of \cite{Aldous-Lyons07}, where a graph is called \emph{unimodular} if its group of automorphisms is unimodular. We feel that in view of the existence of a well-defined modular cocycle $\D_\G$ \eqref{eq:D} on graphs, triviality of this cocycle deserves a separate name, whereas, on the other hand, it does not seem to make much sense to spend two different terms on the same phenomenon (unimodularity of the group of automorphisms). Anyway, for vertex transitive graphs the notions of unimodularity of a graph in the sense of our \dfnref{dfn:dUG} and in the sense of \cite{Aldous-Lyons07} coincide, see \corref{cor:UU} below.
\end{rem}

\begin{prp} \label{prp:mm}
The modular cocycle $\D_\G$ is invariant with respect to the diagonal action of the group of automorphisms $G$ on $\G\times\G$.
\end{prp}

\begin{proof}
Obviously, $G_{gx}=g G_x g^{-1}$ and $G_{gx}\, gy= g G_x\, y$, whence
$$
\card G_{gx}\,gy = \card g G_x\, y = \card G_x\,y \qquad\forall\,x,y\in G \;,
$$
which implies the claim.
\end{proof}

\begin{prp} \label{prp:mg}
For any vertex $x\in\G$ the function
$$
g\mapsto \D_\G(x,gx) \;, \qquad g\in G \;,
$$
is the modular function of the group of automorphisms $G$.
\end{prp}

\begin{proof}
By (\ref{eq:D})
$$
\D_\G(x,gx) = \frac{m(G_x)}{m(G_{gx})} = \frac{m(G_x)}{m(g\, G_x\, g^{-1})} =
\frac{m(G_x)}{m(G_x\,g^{-1})} \;,
$$
where the rightmost term is precisely the modular function on $G$.
\end{proof}

\begin{cor} \label{cor:UU}
A vertex transitive graph is unimodular if and only if its group of automorphisms is unimodular.
\end{cor}

\begin{cor} \label{cor:uu}
Unimodularity of a graph implies unimodularity of its group of automorphisms.
\end{cor}

\begin{rem} \label{rem:i3}
The converse of \corref{cor:uu} is not true, as there exist \emph{finite graphs} (for which the group of automorphisms is obviously finite, hence unimodular) \emph{which are not unimodular}.
\end{rem}

\begin{ex} \label{ex:i3}
The simplest example is provided by the segment graph $\G=I_3$ (here and below the \textsf{segment graph} $I_n$ is the subgraph of $\Z$ which consists of $n$ consecutive points):

\begin{figure}[h]
\begin{center}
     \psfrag{a}[][l]{$1$}
     \psfrag{b}[][l]{$2$}
     \psfrag{c}[][l]{$3$}
          \includegraphics[scale=.5]{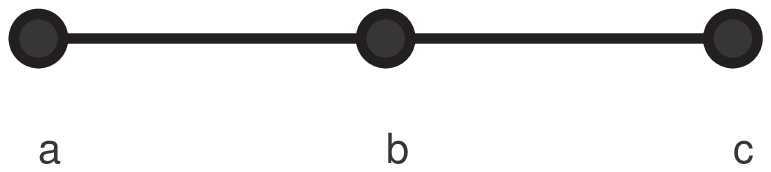}
          \end{center}
\end{figure}

\noindent
Its group of automorphisms $G$ consists of two elements: the identity automorphism, and the inversion $(1,2,3)\mapsto(3,2,1)$, so that the stabilizers $G_1$ and $G_3$ are trivial, whereas $G_2=G$ consists of two elements. Thus, the off-diagonal values of the modular cocycle are
$$
\begin{aligned}
\D_\G(2,1) = \D_\G(2,3) &=& &\!2\;,&   \\
\D_\G(1,2) = \D_\G(3,2) &=& \!\!\!\!\!1&/2\;,&   \\
\D_\G(1,3) = \D_\G(3,1) &=& &\!1\;.&
\end{aligned}
$$
\end{ex}

\begin{rem} \label{rem:rigid}
If a graph $\G$ is \emph{rigid}, then it is obviously \emph{unimodular}. On the other hand, there is a lot of \emph{unimodular non-rigid graphs} as well. The simplest example is provided by the segment graph $I_2$, whose group of automorphisms is isomorphic to $\Z_2$.
\end{rem}

\subsection{Modular function and quotient modular cocycle}

\begin{dfn} \label{dfn:mf}
By \prpref{prp:mm} the modular cocycle $\D_\G$ is constant along $G$-orbits in $\G\times\G$, so that it descends to a function
$$
\D_{\G_\bubu}\left( \ov{(x,y)} \right) \q \D_\G(x,y) \;, \qquad x,y\in\G \;,
$$
on $\G_\bubu$, which we shall call \textsf{modular function}.
\end{dfn}

\begin{rem}
Generally speaking, there is no natural composition operation on $\G_\bubu$, so that the modular function $\D_{\G_\bubu}$ is \emph{not} a cocyle. However, there is a natural \emph{hypercomposition} on $\G_\bubu$, for which the result of composing two elements of $\G_\bubu$ is a probability measure on $\G_\bubu$ rather than a single element (the corresponding structure is a generalization of hypergroups and can be called a \emph{hypergroupoid}, cf. the discussions in \cite{Kaimanovich-Woess95, Kaimanovich-Woess02} in the case of vertex transitive graphs; also see \cite{Kaimanovich05a} for a general discussion of groupoids in this context), and one can look at the behaviour of the modular function $\D_{\G_\bubu}$ with respect to this operation. This is closely related to the relationship between unimodular measures and reversible stationary measures for leafwise simple random walks on $\Gr_\bu$ (cf. \cite{Aldous-Lyons07,Benjamini-Curien12}). We shall return to this topic elsewhere.
\end{rem}

The definitions of the modular cocycle $\D_\G$ \eqref{eq:D} and of the measures $\Nu_\xi$ \eqref{eq:bxy} immediately imply

\begin{prp} \label{prp:dbx}
In terms of the measures $\Nu_\xi$ the modular function $\D_{\G_\bubu}$ takes the form
$$
\D_{\G_\bubu}\left( \ov{(x,y)} \right)
= \frac{\Nu_{\ov x}\left(\ov{(x,y)}\right)}{\Nu_{\ov y}\left(\ov{(y,x)}\right)} \;.
$$
\end{prp}

\begin{dfn} \label{dfn:qmc}
By \prpref{prp:mg}, level sets of the modular cocycle $\D_\G$ are unions of $G$-orbits in $\G$ if and only if the group of automorphisms $G$ is unimodular. In this case $\D_\G$ descends to a cocycle on $\G_\bu$, which we shall denote by $\D_{\G_\bu}$, and call it the \textsf{quotient modular cocycle}.
\end{dfn}

\subsection{Examples}

\begin{ex} \label{ex:i3d}
We have already described the modular cocycle $\D_\G$ for the graph $\G=I_3$ in \exref{ex:i3}. The quotient $\G_\bu$ consists of 2 orbits
\begin{equation} \label{eq:i32}
\ov1 =\{1,3\} \;, \qquad
\ov2 =\{2\} \;,
\end{equation}
and the off-diagonal values of the quotient modular cocycle $\D_{\G_\bu}$ are
\begin{equation} \label{eq:i3d}
\begin{aligned}
\D_{\G_\bu}(\ov1,\ov2) = \D_\G(1,2) &=& \!\!\!\!\!1&/2\;,& \\
\D_{\G_\bu}(\ov2,\ov1) = \D_\G(2,1) &=& &\!2 &
\end{aligned}
\end{equation}
(since $\G$ is finite, its group of automorphisms is finite, in particular, unimodular, so that the quotient modular cocycle is well-defined). Finally, the set $\G_\bubu$ consists of 5 orbits
$$
\begin{aligned}
\ov{(1,1)} &= \{(1,1),(3,3)\} \;, \\
\ov{(1,2)} &= \{(1,2),(3,2)\} \;, \\
\ov{(1,3)} &= \{(1,3),(3,1)\} \;, \\
\ov{(2,1)} &= \{(2,1),(2,3)\} \;, \\
\ov{(2,2)} &= \{(2,2)\} \;.
\end{aligned}
$$
The only values of the modular function $\D_{\G_\bubu}$, which are different from 1, are
$$
\D_{\G_\bubu}\left(\ov{(1,2)}\right) = 1/2 \;, \qquad
\D_{\G_\bubu}\left(\ov{(2,1)}\right) = 2 \;.
$$
\end{ex}

\begin{ex} \label{ex:a}
The \textsf{homogeneous tree} $T=T_d$ \textsf{of degree} $d\ge 3$ is unimodular, so that the modular cocycle, the quotient modular cocycle and the modular function are all identically~1. Note that in this case, since $T$ is a vertex transitive graph, the quotient $T_\bu$ is a singleton, whereas $T_\bubu$ can be identified with $\Z_+=\{0,1,2,\dots\}$ because the orbits $\ov{(x,y)}$ of the action of the group of automorphisms on $T\times T$ are parameterized by the graph distances $\dist(x,y)$ (i.e., the graph $T$ is \textsf{distance transitive}).
\end{ex}

\begin{ex} \label{ex:b}
It is well-known that the groups of automorphisms of the homogeneous trees $T=T_d$ from the previous example contain \emph{non-unimodular subgroups}, the simplest of which is the \textsf{affine group} $\Aff(T)$ \textsf{of the tree} $T$, which consists of all automorphisms which fix a \textsf{reference point} $\o$ from the \textsf{tree boundary} $\p T$ (e.g., see \cite{Cartwright-Kaimanovich-Woess94}). It is not hard to construct a vertex transitive graph whose group of automorphisms is precisely $\Aff(T)$. The simplest such construction is the \textsf{grandfather graph} $T'$ introduced by Trofimov \cite{Trofimov85}: it has the same vertex set as the tree $T$, and its edge set is a union of the edge set of $T$ and of the set of additional edges obtained by joining all pairs of vertices at graph distance 2 on any geodesic issued from the boundary point $\o$. The modular cocycle on $T'$ is
\begin{equation} \label{eq:DT}
\D_{T'}(x,y) = (d-1)^{\b_\o(x,y)} \;,
\end{equation}
where
\begin{equation} \label{eq:MF}
\b_\o(x,y) \q \lim_{z\to\o} \bigl[ \dist(y,z) - \dist(x,z) \bigr]
\end{equation}
is the \textsf{Busemann cocycle} on $T$ determined by the boundary point $\o$. Since the group $\Aff(T)$ is vertex transitive on $T'$, the quotient $T'_\bu$ is a singleton, but the quotient modular cocycle does not exist because of non-unimodularity of $\Aff(T)$. Finally, the orbits $\ov{(x,y)}$ of $\Aff(T)$ on $T\times T$ are parameterized by the graph distance $\dist(x,y)$ and by the \textsf{horodistance} $\b_\o(x,y)$ (see \cite{Cartwright-Kaimanovich-Woess94,Kaimanovich-Woess95} for details), so that $T'_\bubu$ can be identified with $\Z_+\times\Z$ by the map
$$
\ov{(x,y)} \mapsto \bigl( \dist(x,y),\b_\o(x,y) \bigr) \;,
$$
and under this identification the modular function on $T'_\bubu$ is
$$
\D_{T'_\bubu}(m,n) = (d-1)^n \;.
$$
\end{ex}

\section{Invariance, unimodularity and their quasification}

\subsection{Spaces of rooted graphs} \label{sec:Gr}

Let $\Gr_\bu$ be the \textsf{space of isomorphism classes of locally finite (including finite) connected rooted graphs} (i.e., graphs endowed with a distingished vertex). We denote the \textsf{isomorphism class of a rooted graph} $(\G,x)$ by $\ov{(\G,x)}\in\Gr_\bu$ (we have already used this notation for the elements of the orbital quotient \eqref{eq:Gbu} of a single graph). The space $\Gr_\bu$ is endowed with the natural \textsf{projective limit topology}, which makes it a \emph{Polish space}: a sequence $\ov{(\G_n,x_n)}$ converges in $\Gr_\bu$ if and only if for any radius $r>0$ the (pointed) graph metric balls $B_r(\G_n,x_n)$ stabilize. Therefore, $\Gr_\bu$ is a \emph{standard Borel space}, and becomes a \emph{Lebesgue measure space} when endowed with any Borel measure.

Let $\R$ be the \textsf{root moving equivalence relation} on $\Gr_\bu$. Its equivalence classes are at most countable and consist of (isomorphism classes of) all possible rootings of a given graph $\G$, i.e., the $\R$-\textsf{equivalence class} of $\ov{(\G,x)}\in\Gr_\bu$ is precisely the \emph{orbital quotient} of $\G$:
$$
\R\left[\ov{(\G,x)}\right] \q \left\{ \ov{(\G,y)} : y\in\G \right\} \cong \G_\bu
$$
(see \secref{sec:graphs} above). By borrowing from the theory of foliations, we shall sometimes refer to equivalence classes of $\R$ as \textsf{leaves}. The equivalence relation $\R$ is Borel as a subset of $\Gr_\bu\times\Gr_\bu$. Since each $\G_\bu$ is endowed with the graph structure introduced in \secref{sec:graphs}, $\R$ becomes a \emph{graphed equivalence relation} (it is easy to verify that this graph structure is Borel as a subset of $\Gr_\bu\times\Gr_\bu$).

In the same way as $\Gr_\bu$, one can also introduce the \textsf{space $\Gr_\bubu$ of isomorphism classes of doubly rooted graphs} (i.e., isomorphism classes $\ov{(\G,x,y)}$ of graphs $\G$ endowed with two distinguished vertices $x$ and $y$), cf. the definition of $\G_\bubu$ in \secref{sec:graphs}.

Then all the maps defined by diagram \eqref{eq:d2} can be extended from the sets $\G_\bubu, \G_\bu\times\G_\bu$, and $\G_\bu$ corresponding to a single graph $\G$ to the corresponding ``global'' maps between the whole spaces $\Gr_\bubu,\R$, and $\Gr_\bu$, which gives the following commutative diagram:
\begin{equation} \label{eq:d3}
\begin{split}
          \includegraphics{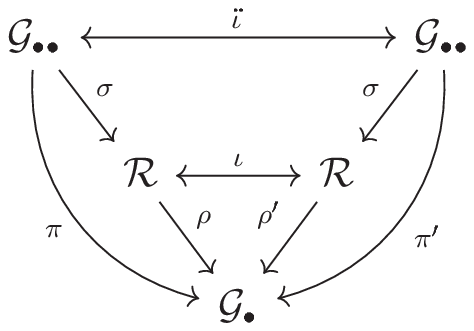}
\end{split}
\end{equation}
The maps appearing on this diagram are denoted in the same way as on diagram \eqref{eq:d1} with omitted subscript $\G$. One can easily verify that all these maps are Borel measurable, and that, moreover, the maps (defined in \secref{sec:diagrams} over individual graphs)
$$
\xi \mapsto \N_\xi \;, \qquad \xi \mapsto \Nu_\xi \;,
$$
which assign to points $\xi\in\Gr_\bu$ the corresponding measures on the fibers $\r^{-1}(\xi)$ and $\pi^{-1}(\xi)$, respectively, are measurable in a natural sense as well. Leafwise modular functions from \dfnref{dfn:mf} will also piece together to provide a global Borel measurable \textsf{modular function} $\D_\bubu$ on $\Gr_\bubu$.

Below we shall need the subsets
$$
\mathring{\Gr}_\bu \subset \bar\Gr_\bu \subset \breve\Gr_\bu \subset \Gr_\bu
$$
which consist of \textsf{all possible rootings of rigid graphs, of unimodular graphs, and of graphs with a unimodular automorphisms group}, respectively. All these subsets are obviously $\R$-saturated, and it is not hard to see that they are all Borel in $\Gr_\bu$.

\begin{rem}
As it follows from \remref{rem:GGu}, the map $\s$ is a bijection over the preimage $\r^{-1}(\mathring{\Gr}_\bu)$ of the set of rigid graphs $\mathring{\Gr}_\bu\subset\Gr_\bu$, and the involutions $\jj,\j$ coincide on $\r^{-1}(\mathring{\Gr}_\bu)$.
\end{rem}

\begin{dfn}
By \prpref{prp:mg} the restriction of the modular function $\D_\bubu$ to $\pi^{-1}(\breve{\Gr}_\bu)$ descends to a multiplicative cocycle $\D_\bu$ of the restriction of the equivalence relation $\R$ to $\breve{\Gr}_\bu$. In other words, $\D_\bu$ is obtained by piecing together the quotient modular cocycles $\D_{\G_\bu}$ (\dfnref{dfn:qmc}) on the equivalence classes from $\breve{\Gr}_\bu$. Again, one can show that $\D_\bu$ is Borel measurable, and we shall call it the \textsf{modular cocycle} on the equivalence relation $\R$ (strictly speaking, on the restriction of $\R$ to $\breve{\Gr}_\bu$).
\end{dfn}

\subsection{Invariant and unimodular measures}

There are two notions of invariance for measures on $\Gr_\bu$. The first one is a specialization of the classical notion of invariance with respect to an equivalence relation due to Feldman and Moore \cite{Feldman-Moore77}: a measure $\mu$ on $\Gr_\bu$ is \textsf{$\R$-invariant} if it is invariant with respect to any \emph{partial transformation} of $\R$. In particular, in the case of \emph{orbit equivalence relations} associated with actions of countable groups this notion coincides with the usual invariance of a measure with respect to a group action. However, it is actually more convenient for us to use an equivalent definition (also due to Feldman and Moore) in terms of counting measures on the equivalence relation $\R$:

\begin{dfn}[\cite{Feldman-Moore77}] \label{dfn:countrel}
The (left) \textsf{counting measure} $\M_\R$ on $\R$ associated with a measure $\mu$ on $\Gr_\bu$ is the measure obtained by integrating the counting measures $\N_\xi$ on the fibers of the projection $\r$ against the measure $\mu$ on the base $\Gr_\bu$:
$$
d\,\M_\R(\xi,\eta) \q d\mu(\xi)\,d\N_\xi(\eta) \;,
$$
or, in a shorter notation,
\begin{equation} \label{eq:MR}
\M_\R \q \int \N_\xi\,d\mu(\xi) \;.
\end{equation}
\end{dfn}

\begin{dfn}[\cite{Feldman-Moore77,Kaimanovich98,Kaimanovich03a}]
A measure $\mu$ on $\Gr_\bu$ is $\R$-\textsf{invariant} if the associated counting measure $\M_\R$ is invariant with respect to the involution $\j$ on $\R$.
\end{dfn}

We shall somewhat reformulate the original definition of the other notion due to Benjamini and Schramm \cite{Benjamini-Schramm01} (also see \cite{Aldous-Steele04,Aldous-Lyons07}) in order to make more transparent its similarity to the above definition of invariant measures. The difference between \dfnref{dfn:countrel} and \dfnref{dfn:U} below is that the equivalence relation $\R$ is replaced with $\Gr_\bubu$, and the system of counting measures $\{\N_\xi\}_{\xi\in\Gr_\bu}$ on the fibers of the projection $\r:\R\to\Gr_\bu$ is replaced with the system of measures $\{\Nu_\xi\}_{\xi\in\Gr_\bu}$ on the fibers of the projection $\pi:\Gr_\bubu\to\Gr_\bu$.

\begin{dfn} \label{dfn:countgr}
The (left) \textsf{counting measure} $\M_{\Gr_\bubu}$ on $\Gr_\bubu$ associated with a measure $\mu$ on $\Gr_\bu$ is the measure obtained by integrating the measures $\Nu_\xi$ on the fibers of the projection $\pi$ against the measure $\mu$ on the base $\Gr_\bu$:
\begin{equation} \label{eq:MU}
\M_{\Gr_\bubu} \q \int \Nu_\xi\,d\mu(\xi) \;.
\end{equation}
\end{dfn}

\begin{dfn} \label{dfn:U}
A measure $\mu$ on $\Gr_\bu$ is called \textsf{unimodular} if the associated counting measure $\M_{\Gr_\bubu}$ is invariant with respect to the involution $\jj$ on $\Gr_\bubu$.
\end{dfn}

\begin{rem} \label{rem:tutte}
As it follows from \remref{rem:GGu}, \emph{the notions of invariance and unimodularity coincide for measures concentrated on the set of rigid graphs} $\mathring{\Gr}_\bu$. Note that this is usually the case for the measures arising from ``natural'' probabilistic constructions (for instance, for the measures associated with the \emph{augmented Galton--Watson trees}, see \cite{Lyons-Pemantle-Peres95,Kaimanovich98,Aldous-Lyons07}). We call this phenomenon \emph{Tutte's principle} as Tutte was apparently the first to describe it for finite graphs when doing his ``census'' of planar triangulations \cite{Tutte62,Tutte80}:
\begin{quote}\small
\dots one feels that for large enough $n$ almost all members of $K'$ with $n$ edges must be unsymmetrical \dots \cite[p. 106]{Tutte80}
\end{quote}
\end{rem}

\subsection{Quasification}

The notion of invariance of a measure with respect to an equivalence relation can be \textsf{quasified} by replacing involution invariance of the associated counting measure with its quasi-invariance, which is a generalization of the usual quasi-invariance with respect to a group action, see \cite{Feldman-Moore77}. Below we shall apply the same idea to the notion of unimodularity as well.

\begin{dfn}
A measure $\mu$ on $\Gr_\bu$ is called \textsf{quasi-invariant} if it is quasi-invariant with respect to any partial transformation of the equivalence relation $\R$, or, equivalently, if the associated counting measure $\M_\R$ on $\R$ is involution quasi-invariant; it is called \textsf{quasi-unimodular} if the associated counting measure $\M_{\Gr_\bubu}$ on $\Gr_\bubu$ is involution quasi-invariant.
\end{dfn}

If a measure $\mu$ is quasi-invariant, then the Radon--Nikodym derivative
$$
\D_\mu (\xi,\eta) \q \frac{d\,\j\M_\R}{d\,\M_\R} (\xi,\eta) \;, \qquad (\xi,\eta)\in\R \;,
$$
satisfies the multiplicative cocycle identity, and is called the \textsf{Radon--Nikodym cocycle} of the measure $\mu$ with respect to the equivalence relation $\R$ \cite{Feldman-Moore77}. Informally the definition of the Radon--Nikodym cocycle is often written as
\begin{equation} \label{eq:dmu}
\D_\mu (\xi,\eta) = \frac{d\,\mu(\eta)}{d\,\mu(\xi)} \;.
\end{equation}

Our main technical tool is

\begin{thm} \label{thm:main}
Any quasi-invariant measure $\mu$ on $\Gr_\bu$ is also quasi-unimodular, and the Radon--Nikodym derivative of the counting measure $\M_{\Gr_\bubu}$ on $\Gr_\bubu$ with respect to the involution $\jj$ is
\begin{equation} \label{eq:RNM}
\frac{d\,\jj\M_{\Gr_\bubu}}{d\,\M_{\Gr_\bubu}}
= \frac{\s^{-1}\D_\mu} {\D_\bubu} \;,
\end{equation}
where $\s^{-1}\D_\mu$ is the $\s$-pullback of the Radon--Nikodym cocycle $\D_\mu$ from $\R$ to $\Gr_\bubu$.
\end{thm}

\begin{proof}[Sketch of the proof]
By the definition of $\M_{\Gr_\bubu}$, for $\th=\left( \ov{(\G,x,y)} \right)\in\Gr_\bubu$
$$
d\,\M_{\Gr_\bubu} (\th) = d\mu \left( \ov{(\G,x)} \right) \, d\Nu_{\ov{(\G,x)}} \left( \ov{(\G,x,y)} \right) \;,
$$
whereas for its involution
$$
d\,\jj\M_{\Gr_\bubu} (\th) =
d\,\M_{\Gr_\bubu} \left( \ov{(\G,y,x)} \right) = d\mu \left( \ov{(\G,y)} \right) \, d\Nu_{\ov{(\G,y)}} \left( \ov{(\G,y,x)} \right) \;,
$$
whence by \eqref{eq:dmu} and \prpref{prp:dbx}
$$
\begin{aligned}
\frac{d\,\jj\M_{\Gr_\bubu}}{d\,\M_{\Gr_\bubu}} (\th)
&= \frac{d\mu \left( \ov{(\G,y)} \right)}{d\mu \left( \ov{(\G,x)} \right)} \cdot
\frac{\Nu_{\ov{(\G,y)}} \left( \ov{(\G,y,x)} \right)}{\Nu_{\ov{(\G,x)}} \left( \ov{(\G,x,y)} \right)} \\
&= \frac{\D_\mu \left( \ov{(\G,x)}, \ov{(\G,y)} \right)}{\D_\bubu\left( \ov{(\G,x,y)} \right)}
= \frac{\D_\mu(\s(\th))}{\D_\bubu(\th)} \;.
\end{aligned}
$$
\end{proof}

\begin{rem} \label{rem:z}
For measures $\mu$ on $\Gr_\bu$, which correspond to stationary non-reversible graphs (with respect to the leafwise simple random walks on $\Gr_\bu$), the Radon--Nikodym cocycle from the left-hand side of \eqref{eq:RNM} was earlier introduced and considered (in a somewhat different form though) by Benjamini and Curien \cite[Section 4]{Benjamini-Curien12}.
\end{rem}

Let $\R^1$ denote the Borel subset of $\R$ which consists of all pairs $\left(\ov{(\G,x)},\ov{(\G,y)}\right)$ such that $\ov x$ and $\ov y$ are neighbours in their common $\R$-equivalence class (endowed with the quotient graph structure introduced in \secref{sec:graphs}), and let $\Gr^1_\bubu=\s^{-1}(\R^1)$ be the subset of $\Gr_\bubu$ which consists of all isomorphism classes of doubly rooted graphs $\ov{(\G,x,y)}$ such that their roots $x$ and $y$ are neighbours in $\G$.

\begin{thm} \label{thm:quasi}
The following conditions on a measure $\mu$ on $\Gr_\bu$ are all equivalent:
\begin{itemize}
\item[(QI)\phantom{$^1$}] $\mu$ is quasi-invariant;
\item[(QI$^1$)] the restriction of the counting measure $\M_\R$ to the set $\R^1\subset\R$ is quasi-invariant with respect to the involution $\j$ on $\R$;
\item[(QU)\phantom{$^1$}] $\mu$ is quasi-unimodular;
\item[(QU$^1$)] the restriction of the counting measure $\M_{\Gr_\bubu}$ to the set $\Gr_\bubu^1\subset\Gr_\bubu$ is quasi-invariant with respect to the involution $\jj$ on $\Gr_\bubu$.
\end{itemize}
\end{thm}

\begin{proof}[Sketch of the proof]
(QI)$\iff$(QI$^1$). This equivalence actually holds for any graphed equivalence relation provided the leafwise graphs are connected. Indeed, (QI)$\implies$(QI$^1$) is obvious, whereas (QI$^1$)$\implies$(QI) follows from yet another general characterization of quasi-invariance of the measure $\mu$ (see \cite{Feldman-Moore77}): it means that for any $\mu$-negligible subset $A\subset\Gr_\bu$ its \textsf{saturation}
$$
\R[A] = \bigcup_{\xi\in A} \R[\xi]
$$
is also $\mu$-negligible. By condition (QI$^1$), if $A$ is $\mu$-negligible, then the union $A^1$ of all neighbours of points from $A$ is also $\mu$-negligible. Since the graph structure on the classes of the equivalence relation $\R$ introduced in \secref{sec:Gr} is such that all leafwise graphs are countable and connected, iteration of this procedure yields $\mu$-negligibility of the whole saturation $\R[A]$.

(QI)$\implies$ (QU). This has been proved in \thmref{thm:main} above.

(QU)$\implies$(QU$^1$). Obvious.

(QU)$\implies$(QI) and (QU$^1$)$\implies$(QI$^1$). The map $\s:\Gr_\bubu\to\R$ consists in collapsing
each of the fibers $\pi^{-1}(\xi),\,\xi\in\Gr_\bu\,,$ of the projection $\pi:\Gr_\bubu\to\Gr_\bu$ onto the corresponding fiber $\r^{-1}(\xi)$ of the projection $\r:\R\to\Gr_\bu$, see diagram \eqref{eq:d3}.

Obviously, the image of the class of each measure $\Nu_\xi$ on $\pi^{-1}(\xi)$ is the class of the measure $\N_\xi$ on $\r^{-1}(\xi)$. Hence, by the definitions of the measures $\M_\R$ on $\R$ and $\M_{\Gr_\bubu}$ on $\Gr_\bubu$ (\dfnref{dfn:countrel} and \dfnref{dfn:countgr}, respectively), the class of the measure $\M_\R$ is the image of the class of the measure $\M_{\Gr_\bubu}$ under the map $\s$. Since the involutions $\jj$ and $\j$ on $\M_{\Gr_\bubu}$ and $\M_\R$, respectively, are semi-conjugate by $\s$, the class of the measure $\j\M_\R$ is the image of the class of the measure $\jj\M_{\Gr_\bubu}$ as well.  Therefore, equivalence of the measures
$\M_{\Gr_\bubu}$ and $\jj\M_{\Gr_\bubu}$ on $\Gr_\bubu$ implies equivalence of the measures $\M_\R$ and $\j\M_\R$ on $\R$. The same argument applied to the restriction of the measure $\M_{\Gr_\bubu}$ to $\Gr_\bubu^1$ provides the implication (QU$^1$)$\implies$(QI$^1$).
\end{proof}

\begin{cor}
Any unimodular measure on $\Gr_\bu$ is quasi-invariant.
\end{cor}

\subsection{Unimodularity and quasi-invariance}

We shall now give a necessary and sufficient condition of unimodularity of a general measure $\mu$ on $\Gr_\bu$ in terms of its Radon--Nikodym cocycle $\D_\mu$ with respect to the equivalence relation $\R$.

\begin{thm} \label{thm:m}
A measure $\mu$ on $\Gr_\bu$ is unimodular if and only if
\begin{itemize}
\item[(i)]
it is concentrated on the set $\breve\Gr_\bu\subset\Gr_\bu$,
\item[(ii)]
it is quasi-invariant with respect to the equivalence relation $\R$, and
\item[(iii)]
its Radon--Nikodym cocycle coincides with the modular cocycle $\D_\bu$ on $\breve\Gr_\bu$.
\end{itemize}
\end{thm}

\begin{proof}
If a measure $\mu$ satisfies conditions (i)--(iii) above, then it is unimodular by \thmref{thm:main}. Conversely, if $\mu$ is unimodular, then it is quasi-invariant by \thmref{thm:quasi}, and by \thmref{thm:main} the $\s$-pullback of its Radon--Nikodym cocycle is (mod 0) the modular function $\D_\bubu$, which means that the set
$$
A = \{ (\xi,\eta)\in\R: \D_\bubu \;\text{is not constant on the fiber}\;\s^{-1}(\xi,\eta) \}
$$
is $\M_\R$-negligible. By quasi-invariance of $\mu$ the $\R$-\textsf{saturation} of $A$ (the union of
all products $\R[\xi]\times\R[\xi],\, \xi\in\Gr_\bu,$ which intersect $A$) is also $\M_\R$-negligible. However, this saturation is precisely $\R\setminus\r^{-1}(\breve\Gr_\bu)$ by the definition of the set $\breve\Gr_\bu$, so that $\mu$ is concentrated on $\breve\Gr_\bu$.
\end{proof}

\begin{cor} \label{cor:iu}
An invariant (resp., unimodular) measure on $\Gr_\bu$ is unimodular (resp., invariant) if and only if it is concentrated on $\bar\Gr_\bu$.
\end{cor}

\begin{cor}[{\cite[Proposition 2.2]{Aldous-Lyons07}}] \label{cor:zz}
A measure $\mu$ is unimodular if and only if the restriction of the counting measure $\M_{\Gr_\bubu}$ to the set $\Gr_\bubu^1\subset\Gr_\bubu$ is invariant with respect to the involution $\jj$ on $\Gr_\bubu$.
\end{cor}

\begin{proof}
Two cocycles on $\R$ coincide if and only if their restrictions to $\R^1$ coincide.
\end{proof}

\begin{rem}
It is easy to give an example of an invariant non-unimodular measure: just take the point measure concentrated on a vertex transitive graph whose group of automorphisms is non-unimodular. Then this measure is obviously invariant (as the corresponding $\R$-class consists just of a single point), but not unimodular. By multiplying this graph by a random rigid graph with an invariant measure one obtains a purely non-atomic example as well. Moreover, in this situation there is no unimodular measure equivalent to the given invariant one. It is also easy to construct a similar example of a unimodular non-invariant measure (just take any invariant$\equiv$unimodular measure on rigid rooted graphs and multiply these graphs by a finite non-unimodular graph). Most likely there also exist examples of \emph{purely non-atomic unimodular measures which are not equivalent to any invariant measure}.
\end{rem}

\subsection{Discrete unimodular measures}

Obviously, a single equivalence class of an equivalence relation carries an invariant probability measure if and only if the class is finite, and in this case the corresponding invariant measure is equidistributed on that class. \thmref{thm:m} completely trivializes the problem of describing the unimodular measures concentrated on the $\R$-equivalence class $\G_\bu$ ($\equiv$ the orbital quotient) of a single graph $\G$.

\begin{dfn}
If the quotient modular cocycle $\D_{\G_\bu}$ on $\G_\bu$ is well-defined (i.e., if the group of automorphisms of $\G$ is unimodular), then we shall say that this cocycle is \textsf{summable} if for any ($\equiv$ certain) $\xi\in\G_\bu$
\begin{equation} \label{eq:sum}
\sum_{\eta\in\G_\bu} \D_{\G_\bu}(\xi,\eta) < \infty \;.
\end{equation}
\end{dfn}

Then \thmref{thm:m} immediately implies

\begin{thm} \label{thm:uf}
The orbital quotient $\G_\bu$ of a connected locally finite graph $\G$ carries a unimodular probability measure (which is then unique) if and only if the group of automorphisms of $\G$ is unimodular, and the quotient modular cocycle $\D_{\G_\bu}$ is summable.
\end{thm}

\begin{cor}[\cite{Benjamini-Lyons-Peres-Schramm99}]
For any finite graph $\G$ the orbital quotient $\G_\bu$ carries a unique unimodular probability measure.
\end{cor}

\begin{rem} \label{rem:nu}
By formula \eqref{eq:DH} condition \eqref{eq:sum} is equivalent to
\begin{equation} \label{eq:m}
\sum_{\xi\in\G_\bu} \frac{1}{\m(\xi)} < \infty \;,
\end{equation}
where $\m(\xi)$ denotes the common value of the left Haar measures $m(G_x)$ of the $G$-stabilizers of the points $x$ from the orbit $\xi\in\G_\bu$. In this form \thmref{thm:uf} was established in \cite[Theorem~3.1]{Aldous-Lyons07} by rather lengthy calculations. It is worth mentioning that condition \eqref{eq:m} also appears in the theory of tree lattices under the name of \textsf{finite covolume condition} as necessary and sufficient for existence of \emph{non-uniform lattices} \cite{Carbone-Rosenberg03} (also see \cite[Section 1.5]{Bass-Lubotzky01} and \cite[Section~2.1]{Carbone01}). We shall return to this relationship elsewhere.
\end{rem}

\begin{ex}
In view of \corref{cor:iu} the simplest example of a finite graph $\G$, for which the invariant and unimodular measures on its orbital quotient $\G_\bu$ are distinct, is provided by the simplest non-unimodular finite graph, i.e., by the segment graph $\G=I_3$. As we have already explained in \exref{ex:i3d}, its orbital quotient $\G_\bu$ consists of two points \eqref{eq:i32}, and the quotient modular cocycle is described by formula \eqref{eq:i3d}. Therefore, the invariant probability measure $\nu$ on $\G_\bu$ is
$$
\nu\left( \ov1 \right) = \nu\left( \ov2 \right) = 1/2 \;,
$$
whereas the unimodular probability measure $\mu$ on $\G_\bu$ is
$$
\mu\left( \ov1 \right) = 2/3 \;, \qquad \mu\left( \ov2 \right) = 1/3 \;.
$$
\end{ex}

\begin{ex}
Let $\G_n$ denote the product graph $I_3\times I_n$ with one ``corner vertex'' removed. Since the graphs $\G_n$ are rigid, the associated orbital quotients coincide with $\G_n$, and the corresponding invariant $\nu_n$ and unimodular $\mu_n$ measures coincide just with the uniform distributions on $\G_n$. Let $\G$ be the limit graph $I_3\times\Z$. Its orbital quotient $\G_\bu$ ($\equiv$ the equivalence class of $\G$ in $\Gr_\bu$) consists of two orbits
$$
\ov{(1,*)} = \{ (1,n): n\in\Z \} \cup \{ (3,n): n\in\Z \}
$$
and
$$
\ov{(2,*)} = \{ (2,n): n\in\Z \} \;.
$$
In the same way as in the previous example, the invariant probability measure $\nu$ on $\G_\bu$ is
$$
\nu\left( \ov{(1,*)} \right) = \nu\left( \ov{(2,*)} \right) = 1/2 \;,
$$
whereas the unimodular probability measure $\mu$ on $\G_\bu$ is
$$
\mu\left( \ov{(1,*)} \right) = 2/3 \;, \qquad \mu\left( \ov{(2,*)} \right) = 1/3 \;.
$$
As $n\to\infty$, the measures $\nu_n=\mu_n$ converge, in the weak$^*$ topology of the space of probability measures on $\Gr_\bu$, to the measure $\mu$, which illustrates the fact that the class of unimodular measures $\U$ is closed in the weak$^*$ topology \cite{Benjamini-Schramm01}, whereas the class of invariant measures $\I$ is not. We shall return to interpreting this phenomenon in terms of the equivalence relation $\R$ by using \thmref{thm:m} later.

\textsf{\begin{figure}[h]
\begin{center}
     \psfrag{a}[][]{2/3}
     \psfrag{c}[][]{1/3}
          \includegraphics[scale=2.]{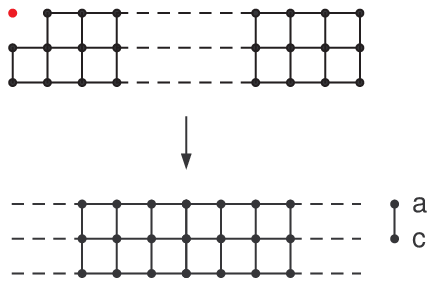}
          \end{center}
\end{figure}}
\end{ex}

\begin{ex}
The summability condition \eqref{eq:sum} can be satisfied even when the orbital quotient $\G_\bu$ is infinite (the counterpart of this fact in the theory of tree lattices is existence of non-uniform lattices, see \remref{rem:nu} above). As pointed out in \cite[Section~3]{Aldous-Lyons07}, the simplest example is provided by \emph{horoballs in homogeneous trees}. Note that this is not the way these graphs are described and called in \cite{Aldous-Lyons07}; its authors ignore that they had been well-known before in geometric and algebraic contexts (e.g., see \cite{Cartwright-Kaimanovich-Woess94,Ronan-Tits99}). More recently these graphs have also been called \emph{canopy trees} after \cite{Aizenman-Warzel06}. Unfortunately, in spite of a clear reference to
\cite{Woess00} in \cite{Aizenman-Warzel06}, the credit for introducing them is sometimes given to \cite{Aizenman-Warzel06} (for instance, in \cite{Benjamini-Lyons-Schramm15}). The picture below is taken from \cite[Section 12.C]{Woess00}.
\begin{figure}[h]
\begin{center}
        \psfrag{w}[][]{$\o$}
        \psfrag{o}[][]{$o$}
        \psfrag{H0}[l][l]{$H_0$}
        \psfrag{H1}[l][l]{$H_{-1}$}
        \psfrag{H2}[l][l]{$H_{-2}$}
        \psfrag{H3}[l][l]{$H_{-3}$}
          \includegraphics[scale=.3]{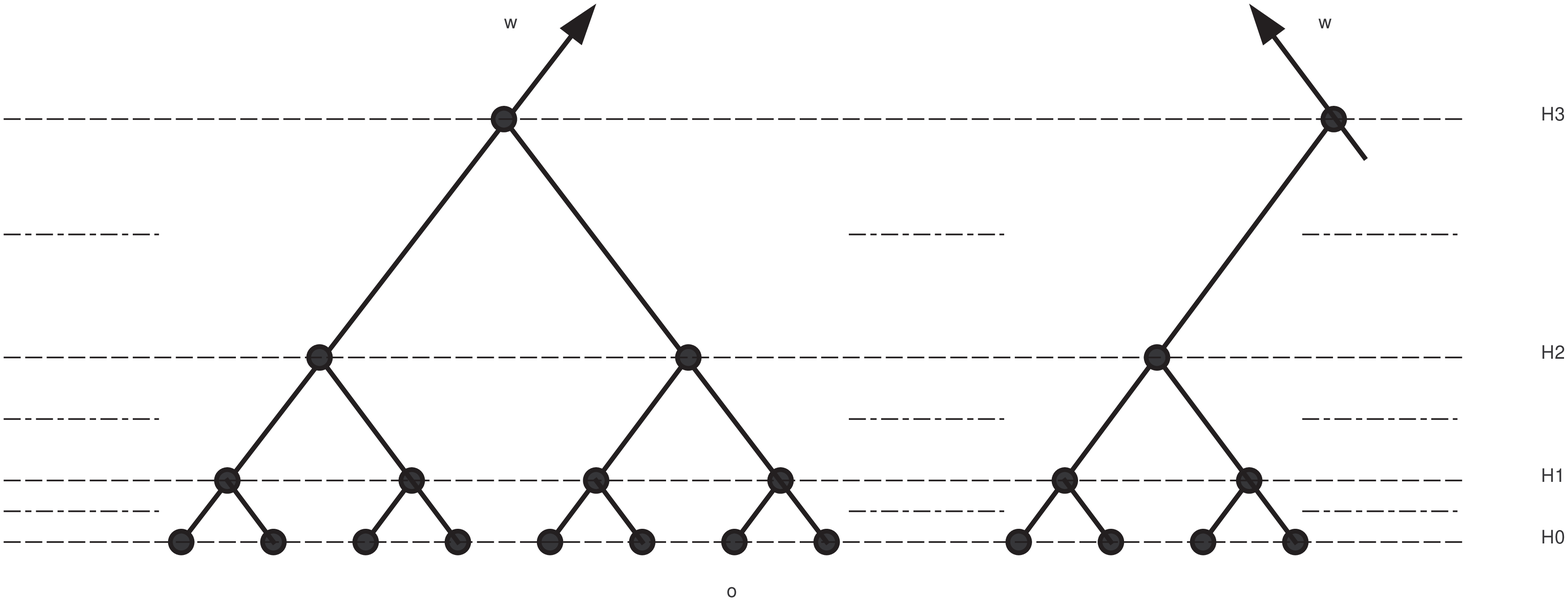}
          \end{center}
\end{figure}
It represents a subgraph $\G$ of the homogeneous tree $T=T_d$ (here $d=3$) obtained by choosing a reference vertex $o\in T$ and a boundary point $\o\in\p T$, and then taking $\G$ to be the union of the \textsf{horospheres}
$$
H_n = \{x\in T: \b_\o(o,x) = n \} \;, \qquad n\in\Z_-=\{0,-1,-2,\dots\} \;,
$$
(see \exref{ex:a} and \exref{ex:b}). Then the orbits of the group of automorphisms of $\G$ are precisely the horospheres $H_n$, so that the orbital quotient $\G_\bu$ can be identified with $\Z_-$. The modular cocycle on $\G$ is then given by the same formula \eqref{eq:DT} as for the grandfather graph. However, unlike for the grandfather graph, in this situation the group of automorphisms of $\G$ is unimodular (because it fixes the horospheres), so that the quotient modular cocycle $\D_{\G_\bu}$ on $\G_\bu$ is well-defined, and
$$
\D_{\G_\bu} (n,n') = (d-1)^{n'-n} \;, \qquad n,n'\in\Z_- \;,
$$
(cf. formula \eqref{eq:MF} for the modular function on the grandfather graph), so that condition \eqref{eq:sum} is obviously satisfied.
\end{ex}

\subsection{Ergodic decomposition of unimodular measures}

Let us remind that given a Lebesgue space $(X,m)$ and a non-singular discrete equivalence relation $R$ on $X$ (i.e., such that the measure $m$ is $R$-quasi-invariant), the measure $m$ is called \textsf{ergodic} with respect to $R$ if any measurable $R$-saturated ($\equiv$ $\R$-invariant) set is either $m$-negligible, or has full measure~$m$. The quotient space of $(X,m)$ determined by the $\s$-algebra of measurable $R$-invariant sets is called the \textsf{space of ergodic components} of the measure $m$ with respect to the equivalence relation $R$. The conditional measures of this map are also $R$-quasi-invariant, and they are called \textsf{ergodic components} of $m$. The arising decomposion of $m$ into an integral of its ergodic components is unique (mod 0). The key property of this decomposition, which we are going to use below, is that the Radon--Nikodym cocycles of the conditional measures coincide with the Radon--Nikodym cocycle of the original measure \cite{Feldman-Moore77}.

The family $\U$ of unimodular measures on $\Gr_\bu$ is convex. The following result gives a description of extreme unimodular measures in terms of ergodicity with respect to the equivalence relation $\R$.

\begin{thm} \label{thm:ex}
A unimodular measure on $\Gr_\bu$ is extreme if and only if it is ergodic as an $\R$-quasi-invariant measure.
\end{thm}

\begin{proof}
By the aforementioned property of the ergodic decomposition, if $\mu$ is not ergodic, then its ergodic components are also unimodular in view of \thmref{thm:m}, so that $\mu$ is not extreme in $\U$. Conversely, if $\mu$ is not extreme in $\U$, then it decomposes as non-trivial convex combination $\mu=\a\mu_1+(1-\a)\mu_2$ of two other unimodular measures $\mu_1,\mu_2$. Since the Radon--Nikodym cocycles of all these measures coincide with the modular cocycle $\D_\bu$, the Radon--Nikodym derivative $d\mu_1/d\mu$ is a globally non-constant and leafwise constant measurable function, which is only possible if $\mu$ is not ergodic.
\end{proof}

As a corollary we immediately obtain

\begin{thm} \label{thm:dec}
Any unimodular measure $\mu$ on $\Gr_\bu$ can be uniquely decomposed into an integral of extreme unimodular measures, and this decomposition coincides with the ergodic decomposition of $\mu$ with respect to the equivalence relation $\R$.
\end{thm}

\begin{rem} \label{rem:zzz}
We believe that the following rather vague definition of the $\s$-algebra $\I$ from \cite[p. 1470]{Aldous-Lyons07} is just a peculiar way (in the absence of an adequate language) of saying that $\I$ is the $\s$-algebra of $\R$-invariant sets:
\begin{quote}\small
Let $\I$ denote the $\s$-field of events (in the Borel $\s$-field of $\Gr_\bu$) that are invariant under non-rooted isomorphisms. To avoid possible later confusion, note that this does not depend on the measure $\mu$, so that even if there are no non-trivial non-rooted isomorphisms $\mu$-a.s., the $\s$-field $\I$ is still not equal (mod 0) to the $\s$-field of $\mu$-measurable sets.
\end{quote}
If so, then \thmref{thm:ex} above coincides with Theorem 4.7 from \cite{Aldous-Lyons07} which was established by using the fact that extremality in $\U$ is equivalent to ergodicity of the associated measure on the path space of the leafwise simple random walk with respect to the time shift. This is analogous to the well-known property of Markov chains with a finite stationary measure: ergodicity of the time shift on the path space is equivalent to irreducibility of the state space \cite{Rosenblatt71,Kaimanovich92}.
\end{rem}

Let us denote by $\hat\Gr_\bu\subset\breve\Gr_\bu$ the union of $\R$-equivalence classes, along which the modular cocycle $\D_\bu$ is summable. Then \thmref{thm:m} in combination with \cite[Theorem 23 and Remark 28]{Kaimanovich10} implies

\begin{thm} \label{thm:hopf}
The discrete ergodic components of a unimodular measure $\mu$ are precisely the ones concentrated on $\R$-equivalence classes from the set $\hat\Gr_\bu$. In other words, the dissipative part of $\mu$ is its restriction to $\hat\Gr_\bu$.
\end{thm}

\bibliographystyle{amsalpha}
\bibliography{arxivsubmission.bbl}

\end{document}